\documentclass[11pt,twoside]{article}
\usepackage{mathrsfs}
\usepackage{amsmath}
\usepackage{amsthm}
\usepackage{amsfonts}
\usepackage{amssymb}
\usepackage{latexsym}
\usepackage[all]{xy}
\usepackage{indentfirst}
\usepackage{colortbl}
\usepackage[numbers,sort&compress]{natbib}
\date{\empty}
\usepackage[normalem]{ulem}

\allowdisplaybreaks[4] \footskip=15pt

\numberwithin{equation}{section} \theoremstyle{plain}
\newtheorem*{thm*}{Main Theorem}
\newtheorem{theorem}{Theorem}[section]
\newtheorem{corollary}[theorem]{Corollary}
\newtheorem*{corollary*}{Corollary}

\newtheorem*{claim*}{Claim}
\newtheorem{lemma}[theorem]{Lemma}
\newtheorem*{lemma*}{Lemma}

\newtheorem*{proposition*}{Proposition}
\newtheorem{remark}[theorem]{Remark}
\newtheorem*{remark*}{Remark}
\newtheorem{example}[theorem]{Example}
\newtheorem*{example*}{Example}

\newtheorem*{question*}{Question}
\newtheorem{definition}[theorem]{Definition}
\newtheorem*{definition*}{Definition}

\newtheorem*{acknowledgements*}{ACKNOWLEDGEMENTS}

\begin{document}
\begin{center}
{\large  \bf The core-EP inverse: A numerical approach for its acute perturbation }
\vspace{0.4cm} {\small \bf Mengmeng  Zhou},\footnote{Mengmeng Zhou (E-mail: mmz9209@163.com): College of Information Engineering, Nanjing Xiaozhuang University, Nanjing 211171, China}
\vspace{0.4cm} {\small \bf Jianlong Chen},\footnote{ Jianlong Chen (E-mail: jlchen@seu.edu.cn): School of Mathematics, Southeast University, Nanjing 210096, China}
\vspace{0.4cm} {\small \bf N\'{e}stor Thome}\footnote{N\'{e}stor Thome (Corresponding author Email: njthome@mat.upv.es): Instituto Universitario de Matem\'{a}tica Multidisciplinar, Universitat Polit\`{e}cnica de Val\`{e}ncia, Valencia 46022, Spain}

\end{center}

\bigskip

{ \bf  Abstract:}  \leftskip0truemm\rightskip0truemm 
This paper studies the concept of stable perturbation $B\in\mathbb{C}^{n\times n}$ for the core-EP inverse of a matrix $A\in\mathbb{C}^{n\times n}$ with index $k$. 
For a given stable perturbation $B$ of $A$,  explicit expressions of its core-EP inverse $B^{\scriptsize\textcircled{\tiny $\dagger$}}$ and its projection at zero $B^\pi$ are presented. 
Then, the perturbation bounds of $\parallel B^{\scriptsize\textcircled{\tiny $\dagger$}}-A^{\scriptsize\textcircled{\tiny $\dagger$}}\parallel/\parallel A^{\scriptsize\textcircled{\tiny $\dagger$}}\parallel$ and $\parallel B^{\pi}-A^{\pi}\parallel$ are given provided that $B$ is a stable perturbation of $A$. 
In addition, we investigate the concept of acute perturbation of $A$. We give a perturbation analysis with respect to core-EP inverses. We provide a condition under which the acute perturbation coincides with the stable perturbation for core-EP inverses.


 \bigskip

\section {\bf Introduction}
The set of all $m\times n$ complex matrices will be denoted by $\mathbb{C}^{m\times n}$. For $A\in\mathbb{C}^{m\times n}$, the notations $R(A)$, $N(A)$, ${\rm rk}(A)$ and $A^{\ast}$ stand for the range space, the null space, the rank and the conjugate transpose of $A$, respectively. The identity matrix of an appropriate order is denoted by $I$. The symbols $\parallel \cdot\parallel$ and $\rho(\cdot)$ denote the spectral norm and spectral radius, respectively.

The (unique) matrix $A^{\dagger}\in\mathbb{C}^{n\times m}$ satisfying the following four equations \textcolor[rgb]{0.00,0.00,1.00}{\cite{P}}
$$AA^{\dagger}A=A,~A^{\dagger}AA^{\dagger}=A^{\dagger},~(AA^{\dagger})^{\ast}=AA^{\dagger},~(A^{\dagger}A)^{\ast}=A^{\dagger}A,$$
is called the Moore-Penrose inverse of $A\in\mathbb{C}^{m\times n}$. The (unique) matrix $A^{D}\in\mathbb{C}^{n\times n}$ is called the Drazin inverse of $A\in\mathbb{C}^{n\times n}$ if it satisfies the following three equations \textcolor[rgb]{0.00,0.00,1.00}{\cite{D}}
$$A^{k+1}A^{D}=A^{k},~A^{D}AA^{D}=A^{D},~AA^{D}=A^{D}A,~{\rm for ~some~integer~}k.$$
If $A$ is singular and $k$ is the smallest positive integer such that ${\rm rk}(A^{k+1})={\rm rk}(A^{k})$ holds, then $k$ is called the index of $A\in\mathbb{C}^{n\times n}$ and denoted by ${\rm ind}(A).$ When ${\rm ind}(A)=1,$ the Drazin inverse is called the group inverse and denoted by $A^{\#}$.

The core-EP inverse of $A\in\mathbb{C}^{n\times n}$ with ${\rm ind}(A)=k$ is the unique matrix $A^{\scriptsize\textcircled{\tiny $\dagger$}}\in\mathbb{C}^{n\times n}$ satisfying the following three equations \textcolor[rgb]{0.00,0.00,1.00}{ \cite{GC, MPM}}
$$A^{\scriptsize\textcircled{\tiny $\dagger$}}A^{k+1}=A^{k},~A(A^{\scriptsize\textcircled{\tiny $\dagger$}})^{2}=A^{\scriptsize\textcircled{\tiny $\dagger$}},~(AA^{\scriptsize\textcircled{\tiny $\dagger$}})^{\ast}=AA^{\scriptsize\textcircled{\tiny $\dagger$}}.$$
 We denote $A^{\pi}:=I-AA^{\scriptsize\textcircled{\tiny $\dagger$}}$. By definition of the core-EP inverse, it is known that $(A^{\pi})^{2}=A^{\pi}=(A^{\pi})^{\ast}$. That is, $A^{\pi}$ is an orthogonal projector. 
If $k=1$, then the core-EP inverse is reduced to the core inverse of $A$ and denoted by $A^{\scriptsize\textcircled{\tiny \#}}$ \cite{BT, XCZ}, i.e., the core inverse of $A\in\mathbb{C}^{n\times n}$ with  ${\rm ind}(A)=1$ is the unique matrix $A^{\scriptsize\textcircled{\tiny \#}}\in\mathbb{C}^{n\times n}$ satisfying the following three equations
$$A^{\scriptsize\textcircled{\tiny \#}}A^{2}=A,~A(A^{\scriptsize\textcircled{\tiny \#}})^{2}=A^{\scriptsize\textcircled{\tiny \#}},~(AA^{\scriptsize\textcircled{\tiny \#}})^{\ast}=AA^{\scriptsize\textcircled{\tiny \#}}.$$
Some publications related to the core-EP inverse are \cite{FLT, MO, MO1, MO2, W, ZC, ZCLW}.

In 1973, Wedin \cite{W0} introduced the concept of acute perturbation of the Moore-Penrose inverse. Let $A,E\in\mathbb{C}^{m\times n}$ and $B=A+E$. The matrix $B$ is an acute perturbation with respect to the Moore-Penrose inverse if
$$\parallel BB^{\dagger}-AA^{\dagger}\parallel<1~{\rm and}~\parallel B^{\dagger}B-A^{\dagger}A\parallel<1,$$
in which case, $A$ and $B$ are acute. In 1990, Stewart et al. \cite{SS} proved that $\parallel BB^{\dagger}-AA^{\dagger}\parallel<1$ if and only if
$$R(A)\cap N(B^{\ast})=\{0\}~{\rm and}~R(B)\cap N(A^{\ast})=\{0\}.$$

Many papers have focused on investigating explicit expressions for the Drazin inverse, providing the related norm upper bounds based on these various formulas and giving the corresponding error estimations \cite{CKW, CRV, K, LW, MG, RW, W2, W3, WW, YD, ZW}. Let $A,B\in\mathbb{C}^{n\times n}$ with ${\rm ind}(A)=k$ and ${\rm ind}(B)=s$. Xu et al. \cite{XSW} defined that $B$ is a stable perturbation of $A$ with respect to the Drazin inverse if $B$ satisfies the condition $(C_{s})$: $R(A^{k})\cap N(B^{s})=\{0\}$ and $R(B^{s})\cap N(A^{k})=\{0\}$.  In \cite{W1}, Wei introduced the concept of acute perturbation with respect to the group inverse. That is, $B$ is an acute perturbation of $A$ with respect to the group inverse if $\parallel B-A\parallel$ is small and the spectral radius $\rho(BB^{\#}-AA^{\#})<1$. The author proved that the acute perturbation concept coincides with that of  stable perturbation of the group inverse if condition $(C_{1})$ holds. Furthermore, Qiao et al. \cite{QW} considered the concept of  acute perturbation with respect to the Drazin inverse and oblique projectors (For more details on spectral projectors and generalized inverses see \cite{APST, BG, D1, S}). They presented an example to show that the spectral radius is a better choice than the spectral norm with respect to the Drazin inverse (group inverse). They also proved that $B$ is an acute perturbation of $A$ if and only if $B$ satisfies the condition $(C_{s})$. Perturbation analysis involving generalized inverses, projections, their applications to the study of linear systems, and related problems has been studied in the literature from different points of view as we can see in the following references \cite{BMP, XWG}.

Recently, Ma \cite{M} studied optimal perturbation bounds of core inverses. Moreover, Ma et al. \cite{MS} investigated optimal perturbation bounds of  core-EP inverses, generalizing the results in \cite{M}. 
Ma \cite{M1} also pointed out the difficulty involved in investigations on 
stable or acute perturbations for weighted core-EP inverses. For that,  for a given matrix $A\in\mathbb{C}^{m\times n}$, it is necessary to consider, among other things, the influence of the weighted matrix $W\in\mathbb{C}^{n\times m}$ on the index of $AW$ and $WA$, the conditions under which the weighted core-EP inverse satisfies a stable perturbation, and the solution to one of its stable perturbations. Let $A ,B\in\mathbb{C}^{n\times n}$ with ${\rm ind}(A)=k$ and ${\rm ind}(B)=s$. Zhou et al. \cite{ZCT} characterized a class of matrices related to the core-EP inverse under the condition that:
$$(C_{s,\ast})~~R(A^{k})\cap N((B^{s})^{\ast})=\{0\}~{\rm and}~R(B^{s})\cap N((A^{k})^{\ast})=\{0\}.$$
When the condition $(C_{s,\ast})$ holds, they proved that $I+(L_{B}-A)A^{\scriptsize\textcircled{\tiny $\dagger$}}$ is nonsingular and presented the explicit expression of $B^{\scriptsize\textcircled{\tiny $\dagger$}}$, where $L_{B}=B^{2}B^{\scriptsize\textcircled{\tiny $\dagger$}}$. The upper bound of $\parallel B^{\scriptsize\textcircled{\tiny $\dagger$}}-A^{\scriptsize\textcircled{\tiny $\dagger$}}\parallel/\parallel A^{\scriptsize\textcircled{\tiny $\dagger$}}\parallel$ was given under the assumption that $(C_{s,\ast})$ is satisfied and $${\rm max}\{\parallel(L_{B}-A)A^{\scriptsize\textcircled{\tiny $\dagger$}}\parallel, \parallel(A^{\scriptsize\textcircled{\tiny $\dagger$}})^{\ast}(L_{B}-A^{\ast})\parallel\}<\frac{1}{1+\sqrt{\parallel A^{\pi}\parallel}}.$$

Motivated by above discussion, we introduce the notion of the stable perturbation and the acute perturbation with respect to the core-EP inverse (separately) when the condition $(C_{s,\ast})$ is satisfied. Then, we investigate the relationship between them.

The paper is organized as follows. In Section 2, we give some preliminary lemmas. In Section 3, we define the stable perturbation of the core-EP inverse and present equivalent characterization of the stable perturbation according to the results in \cite{ZCT}. Moreover, we obtain new expressions for $B^{\scriptsize\textcircled{\tiny $\dagger$}}$ and $B^{\pi}$ under condition of stable perturbation. In addition, the upper bounds for $\parallel B^{\scriptsize\textcircled{\tiny $\dagger$}}-A^{\scriptsize\textcircled{\tiny $\dagger$}}\parallel/\parallel A^{\scriptsize\textcircled{\tiny $\dagger$}}\parallel$ and $\parallel B^{\pi}-A^{\pi}\parallel$ are obtained, respectively. In Section 4, we introduce the acute perturbation for the core-EP inverse. Then, some characterizations of the acute perturbation are presented. Moreover, sufficient and necessary conditions for the acute perturbation of the core-EP inverse are derived. An numerical example is provided to illustrate the validity of the acute perturbation for the core-EP inverse. In Section 5, we show that the perturbation results in \cite{M, MS} are acute perturbations.

\section{\bf Preliminaries}\label{a}
Let $B\in\mathbb{C}^{n\times n}$ with ${\rm ind}(B)=s$. We write $L_{B}:=B^{2}B^{\scriptsize\textcircled{\tiny $\dagger$}}.$ By \cite{FLT}, we have ${\rm rk}(L_{B})={\rm rk}(B^{s}).$ Next, some auxiliary lemmas are given.
\begin{lemma}\label{a0}\emph{\cite{W}} (Core-EP decomposition) Let $A\in \mathbb{C}^{n\times n}$ with ${\rm ind}(A)=k.$ Then $A$ can be uniquely written as $A=A_{1}+A_{2}$, where
\begin{itemize}
\item [{\rm (i)}]${\rm ind}(A_{1})\leq 1$;
\item [{\rm (ii)}] $A_{2}^{k}=0$;
\item [{\rm (iii)}] $A_{1}^{\ast}A_{2}=A_{2}A_{1}=0$.
\end{itemize}
Moreover, there exists a unitary matrix $U\in \mathbb{C}^{n\times n}$ such that
$$A_{1}=U\left(\begin{matrix}
T&S\\
0&0
\end{matrix}
\right)U^{\ast},~~A_{2}=U\left(\begin{matrix}
0&0\\
0&N
\end{matrix}
\right)U^{\ast},$$
where $T\in \mathbb{C}^{r\times r}$ is nonsingular, $N$ is nilpotent and ${\rm rk}(A^{k})=r.$
\end{lemma}

For $A\in \mathbb{C}^{n\times n}$ being as in Lemma \ref{a0}, it is known \cite{W} that
$$A^{\scriptsize\textcircled{\tiny $\dagger$}}=U\left(\begin{matrix}
T^{-1}&0\\
0&0
\end{matrix}
\right)U^{\ast},~~A^{\pi}=U\left(\begin{matrix}
0&0\\
0&I
\end{matrix}
\right)U^{\ast}.$$

\begin{lemma}\label{a00}\emph{\cite{GC}} Let $A\in \mathbb{C}^{n\times n}$ with ${\rm ind}(A)=k.$ Then the following statements hold:
\begin{itemize}
\item [{\rm (i)}] $AA^{\scriptsize\textcircled{\tiny $\dagger$}}=A^{m}(A^{\scriptsize\textcircled{\tiny $\dagger$}})^{m}$, for arbitrary positive integer $m$;
\item [{\rm (ii)}] $A^{\scriptsize\textcircled{\tiny $\dagger$}}=A^{D}A^{k}(A^{k})^{\dagger}$;
\item [{\rm (iii)}] $(A^{\scriptsize\textcircled{\tiny $\dagger$}})^{\scriptsize\textcircled{\tiny $\dagger$}}=(A^{\scriptsize\textcircled{\tiny $\dagger$}})^{\scriptsize\textcircled{\tiny \#}}=A^{2}A^{\scriptsize\textcircled{\tiny $\dagger$}}$;
\item [{\rm (iv)}] $((A^{\scriptsize\textcircled{\tiny $\dagger$}})^{\scriptsize\textcircled{\tiny $\dagger$}})^{\scriptsize\textcircled{\tiny $\dagger$}}=A^{\scriptsize\textcircled{\tiny $\dagger$}}$.
\end{itemize}
\end{lemma}

\begin{lemma}\label{a2}\emph{\cite{ZCT}} Let $B_{1}\in\mathbb{C}^{m\times m}$ be nonsingular and let $P\in\mathbb{C}^{m\times n}$ and $Q\in\mathbb{C}^{n\times m}$ be arbitrary matrices. Then, the matrix $W:=\left(\begin{matrix}
B_{1}&B_{1}P\\
QB_{1}&QB_{1}P
\end{matrix}
\right)\in \mathbb{C}^{(m+n)\times (m+n)}$ is core invertible if and only if $I+PQ$ is nonsingular. In this case,
$$W^{\scriptsize\textcircled{\tiny \#}}=\left(\begin{matrix}
((I+Q^{\ast}Q)B_{1}(I+PQ))^{-1}&((I+Q^{\ast}Q)B_{1}(I+PQ))^{-1}Q^{\ast}\\
Q((I+Q^{\ast}Q)B_{1}(I+PQ))^{-1}&Q((I+Q^{\ast}Q)B_{1}(I+PQ))^{-1}Q^{\ast}
\end{matrix}
\right),$$
\begin{eqnarray*}
  WW^{\scriptsize\textcircled{\tiny \#}} &=& \left(\begin{matrix}
(I+Q^{\ast}Q)^{-1}&(I+Q^{\ast}Q)^{-1}Q^{\ast}\\
Q(I+Q^{\ast}Q)^{-1}&Q(I+Q^{\ast}Q)^{-1}Q^{\ast}
\end{matrix}
\right),
\end{eqnarray*}
and
\begin{eqnarray*}
 W^{\pi}  &=&\left(\begin{matrix}
I-(I+Q^{\ast}Q)^{-1}&-(I+Q^{\ast}Q)^{-1}Q^{\ast}\\
-Q(I+Q^{\ast}Q)^{-1}&I-Q(I+Q^{\ast}Q)^{-1}Q^{\ast}
\end{matrix}\right).
\end{eqnarray*}
\end{lemma}

\section{Stable perturbation of the core-EP inverse}\label{b}
In this section, let $A\in\mathbb{C}^{n\times n}$ with ${\rm ind}(A)=k>0$. For any $B\in\mathbb{C}^{n\times n}$ with ${\rm ind}(B)=s,$ let $L_{B}=B^{2}B^{\scriptsize\textcircled{\tiny $\dagger$}}$. If $B\in\mathbb{C}^{n\times n}$ satisfies condition $(C_{s,\ast})$, by \cite[Lemma 3.1]{ZCT}, it is known that $I+(A^{\scriptsize\textcircled{\tiny $\dagger$}})^{\ast}(L_{B}-A^{\ast})$ and $I+(L_{B}-A)A^{\scriptsize\textcircled{\tiny $\dagger$}}$ are nonsingular.

Firstly, we give the definition and equivalent conditions of the stable perturbation
\begin{definition}\label{b1} Let $A\in\mathbb{C}^{n\times n}$. A matrix $B\in\mathbb{C}^{n\times n}$ is said to be a stable perturbation of $A$ with respect to the core-EP inverse (in short, stable perturbation) if $I-(B^{\pi}-A^{\pi})^{2}$ is nonsingular.
\end{definition}

Now, we present a characterization of stable perturbations.
\begin{lemma}\label{b2} Let $A\in \mathbb{C}^{n\times n}$ with ${\rm ind}(A)=k>0.$ Then the following conditions on $B\in \mathbb{C}^{n\times n}$ with ${\rm ind}(B)=s$ are equivalent:
\begin{itemize}
\item [{\rm (i)}] $B$ is a stable perturbation of $A$;
\item [{\rm (ii)}] $L_{B}$ is a stable perturbation of $A$;
\item [{\rm (iii)}] ${\rm rk}(B^{s})={\rm rk}(A^{k})={\rm rk}((A^{k})^{\ast}L_{B}A^{k})$;
\item [{\rm (iv)}] $B$ satisfies condition $(C_{s,\ast})$;
\item [{\rm (v)}] $I+(L_{B}-A)A^{\scriptsize\textcircled{\tiny $\dagger$}}$ is nonsingular, $A^{\pi}(I+(L_{B}-A)A^{\scriptsize\textcircled{\tiny $\dagger$}})^{-1}L_{B}=0$;
\item [{\rm (vi)}] If $A$ is written as in Lemma \ref{a0}, then $L_{B}$ has the following matrix form:
$$L_{B}=U\left(\begin{matrix}
B_{1}&B_{1}P\\
QB_{1}&QB_{1}P
\end{matrix}
\right)U^{\ast},$$ for some matrices $B_{1}$, $P$ and $Q$ such that $B_{1}$ and $I+PQ$ are nonsingular;
\item [{\rm (vii)}] ${\rm rk}(B^{s})={\rm rk}(A^{k})$, $I+(L_{B}-A)A^{\scriptsize\textcircled{\tiny $\dagger$}}$ is nonsingular.
\end{itemize}
\end{lemma}

\begin{proof} ${\rm (i)}\Leftrightarrow{\rm (ii)}:$ From Lemma \ref{a00} and $L_{B}=B^{2}B^{\scriptsize\textcircled{\tiny $\dagger$}}$, we have $L_{B}(L_{B})^{\scriptsize\textcircled{\tiny $\dagger$}}=BB^{\scriptsize\textcircled{\tiny $\dagger$}}$. So, by Definition \ref{b1}, we obtain $I-(B^{\pi}-A^{\pi})^{2}$ is nonsingular if and only if $I-(L_{B}^{\pi}-A^{\pi})^{2}$ is nonsingular.

By Definition \ref{b1} and \cite[Theorem 4.1 and Theorem 4.2]{ZCT}, we know that ${\rm (i)}\Leftrightarrow{\rm (iii)}\Leftrightarrow{\rm (iv)}\Leftrightarrow{\rm (v)}\Leftrightarrow{\rm (vi)}\Leftrightarrow{\rm (vii)}.$
\end{proof}

\begin{remark}\label{b2223}  According to the proof  of \cite[Theorem 4.2]{ZCT}, we know that $L_{B}$ depends on the choice of $B$ and not on the choice of $B^{\scriptsize\textcircled{\tiny $\dagger$}}$ in Lemma \ref{b2}.
\end{remark}

Next, we give characterizations of the core-EP inverse when $B$ is a stable perturbation of $A$. 
We denote
$$
E_{B}=L_{B}-A, \qquad F_{B}=L_{B}-A^{\ast}, \qquad 
Y=(I+(A^{\scriptsize\textcircled{\tiny $\dagger$}})^{\ast}F_{B})^{-1}(A^{\scriptsize\textcircled{\tiny $\dagger$}})^{\ast}F_{B}A^{\pi}
$$
and
$$
Z=A^{\pi}E_{B}A^{\scriptsize\textcircled{\tiny $\dagger$}}(I+E_{B}A^{\scriptsize\textcircled{\tiny $\dagger$}})^{-1}.
$$
\begin{lemma}\label{b3} Let $A\in\mathbb{C}^{n\times n}$ with ${\rm ind}(A)=k>0$. If $B\in \mathbb{C}^{n\times n}$ is a stable perturbation of $A$ with ${\rm ind}(B)=s$, then $I+YZ$ is nonsingular.
\end{lemma}

\begin{proof} Suppose that $A$ is written as in Lemma \ref{a0}. By Lemma \ref{b2} ${\rm (i)}$ and ${\rm (vi)}$, we have $L_{B}=U\left(\begin{matrix}
B_{1}&B_{1}P\\
QB_{1}&QB_{1}P
\end{matrix}
\right)U^{\ast},$ for some matrices $B_{1}$, $P$ and $Q$ such that $B_{1}$ and $I+PQ$ are nonsingular. It then follows,
$$E_{B}=U\left(\begin{matrix}
B_{1}-T&B_{1}P-S\\
QB_{1}&QB_{1}P-N
\end{matrix}
\right)U^{\ast},~~F_{B}=U\left(\begin{matrix}
B_{1}-T^{\ast}&B_{1}P\\
QB_{1}-S^{\ast}&QB_{1}P-N^{\ast}
\end{matrix}
\right)U^{\ast},$$ $$I+E_{B}A^{\scriptsize\textcircled{\tiny $\dagger$}}=U\left(\begin{matrix}
B_{1}T^{-1}&0\\
QB_{1}T^{-1}&I
\end{matrix}
\right)U^{\ast},~~I+(A^{\scriptsize\textcircled{\tiny $\dagger$}})^{\ast}F_{B}=U\left(\begin{matrix}
(T^{-1})^{\ast}B_{1}&(T^{-1})^{\ast}B_{1}P\\
0&I
\end{matrix}
\right)U^{\ast}.$$ By a direct computation, we get that
$$(I+E_{B}A^{\scriptsize\textcircled{\tiny $\dagger$}})^{-1}=U\left(\begin{matrix}
TB_{1}^{-1}&0\\
-Q&I
\end{matrix}
\right)U^{\ast},~~(I+(A^{\scriptsize\textcircled{\tiny $\dagger$}})^{\ast}F_{B})^{-1}=U\left(\begin{matrix}
B_{1}^{-1}T^{\ast}&-P\\
0&I
\end{matrix}
\right)U^{\ast},$$
$$Y=(I+(A^{\scriptsize\textcircled{\tiny $\dagger$}})^{\ast}F_{B})^{-1}(A^{\scriptsize\textcircled{\tiny $\dagger$}})^{\ast}F_{B}A^{\pi}=U\left(\begin{matrix}
0&P\\
0&0
\end{matrix}
\right)U^{\ast},$$ $$Z=A^{\pi}E_{B}A^{\scriptsize\textcircled{\tiny $\dagger$}}(I+E_{B}A^{\scriptsize\textcircled{\tiny $\dagger$}})^{-1}=U\left(\begin{matrix}
0&0\\
Q&0
\end{matrix}
\right)U^{\ast}.$$ Since $I+PQ$ is nonsingular and $I+YZ=U\left(\begin{matrix}
I+PQ&0\\
0&I
\end{matrix}
\right)U^{\ast},$ we obtain that $I+YZ$ is nonsingular.
\end{proof}

From the above lemma, we obtain representations of $B^{\scriptsize\textcircled{\tiny $\dagger$}}$ in terms of $A^{\scriptsize\textcircled{\tiny $\dagger$}}$ and of $B^{\pi}$ in terms of $A^{\pi}$ and $A^{\scriptsize\textcircled{\tiny $\dagger$}}$.
\begin{theorem}\label{b4}  Let $A\in\mathbb{C}^{n\times n}$ with ${\rm ind}(A)=k>0$. If $B\in \mathbb{C}^{n\times n}$ is a stable perturbation of $A$ with ${\rm ind}(B)=s$, then
$$B^{\scriptsize\textcircled{\tiny $\dagger$}}=W_{1}^{-1}A^{\scriptsize\textcircled{\tiny $\dagger$}}(I+E_{B}A^{\scriptsize\textcircled{\tiny $\dagger$}})^{-1}W_{2}^{-1},$$
$$B^{\pi}=W_{2}A^{\pi}(I+E_{B}A^{\scriptsize\textcircled{\tiny $\dagger$}})^{-1}W_{2}^{-1},$$
where $W_{1}=(I+YZ)(I-Z)$ with $W_{1}^{-1}=(I+Z)(I+YZ)^{-1}$ and $W_{2}=(I-Z^{\ast})(I+Z^{\ast}Z)$ with $W_{2}^{-1}=(I+Z^{\ast}Z)^{-1}(I+Z^{\ast}).$
\end{theorem}

\begin{proof}
	We know that
	$E_{B}=L_{B}-A,$ $F_{B}=L_{B}-A^{\ast},$ $Y=(I+(A^{\scriptsize\textcircled{\tiny $\dagger$}})^{\ast}F_{B})^{-1}(A^{\scriptsize\textcircled{\tiny $\dagger$}})^{\ast}F_{B}A^{\pi}$ and $Z=A^{\pi}E_{B}A^{\scriptsize\textcircled{\tiny$\dagger$}}(I+E_{B}A^{\scriptsize\textcircled{\tiny $\dagger$}})^{-1}.$ Let $A$ be as in Lemma \ref{a0}. By Lemma \ref{b2}, we have $L_{B}=U\left(\begin{matrix}
B_{1}&B_{1}P\\
QB_{1}&QB_{1}P
\end{matrix}
\right)U^{\ast},$ for some matrices $B_{1}$, $P$ and $Q$ such that $B_{1}$ and $I+PQ$ are nonsingular. By the proof of Lemma \ref{b3}, we obtain $$W_{1}=(I+YZ)(I-Z)=U\left(\begin{matrix}
I+PQ&0\\
-Q&I
\end{matrix}
\right)U^{\ast},$$ $$W_{1}^{-1}=U\left(\begin{matrix}
(I+PQ)^{-1}&0\\
Q(I+PQ)^{-1}&I
\end{matrix}
\right)U^{\ast}=(I+Z)(I+YZ)^{-1},$$
$$W_{2}=(I-Z^{\ast})(I+Z^{\ast}Z)=U\left(\begin{matrix}
I+Q^{\ast}Q&-Q^{\ast}\\
0&I
\end{matrix}
\right)U^{\ast},$$ $$W_{2}^{-1}=U\left(\begin{matrix}
(I+Q^{\ast}Q)^{-1}&(I+Q^{\ast}Q)^{-1}Q^{\ast}\\
0&I
\end{matrix}
\right)U^{\ast}=(I+Z^{\ast}Z)^{-1}(I+Z^{\ast}).$$
Again, by the proof of Lemma \ref{b3}, Lemma \ref{a00} and Lemma \ref{a2}, we have
\begin{eqnarray*}
\begin{split}
  & W_{1}^{-1}A^{\scriptsize\textcircled{\tiny $\dagger$}}(I+E_{B}A^{\scriptsize\textcircled{\tiny $\dagger$}})^{-1}W_{2}^{-1}=U\left(\begin{matrix}
(I+PQ)^{-1}&0\\
Q(I+PQ)^{-1}&I
\end{matrix}
\right)\left(\begin{matrix}
T^{-1}&0\\
0&0
\end{matrix}
\right)\left(\begin{matrix}
TB_{1}^{-1}&0\\
-Q&I
\end{matrix}
\right)\\& \times\left(\begin{matrix}
(I+Q^{\ast}Q)^{-1}&(I+Q^{\ast}Q)^{-1}Q^{\ast}\\
0&I
\end{matrix}
\right)U^{\ast}
   = U\left(\begin{matrix}
(I+PQ)^{-1}&0\\
Q(I+PQ)^{-1}&I
\end{matrix}
\right)\times\\&\left(\begin{matrix}
B_{1}^{-1}&0\\
0&0
\end{matrix}
\right)\left(\begin{matrix}
(I+Q^{\ast}Q)^{-1}&(I+Q^{\ast}Q)^{-1}Q^{\ast}\\
0&I
\end{matrix}
\right)U^{\ast}\\&=U\left(\begin{matrix}
((I+Q^{\ast}Q)B_{1}(I+PQ))^{-1}&((I+Q^{\ast}Q)B_{1}(I+PQ))^{-1}Q^{\ast}\\
Q((I+Q^{\ast}Q)B_{1}(I+PQ))^{-1}&Q((I+Q^{\ast}Q)B_{1}(I+PQ))^{-1}Q^{\ast}
\end{matrix}
\right)U^{\ast}\\&=L_{B}^{\scriptsize\textcircled{\tiny \#}}=B^{\scriptsize\textcircled{\tiny $\dagger$}}
\end{split}
\end{eqnarray*}
and by the proof of Lemma \ref{b3} we get
\begin{eqnarray*}
  W_{2}A^{\pi}(I+E_{B}A^{\scriptsize\textcircled{\tiny $\dagger$}})^{-1}W_{2}^{-1} &=&U\left(\begin{matrix}
I+Q^{\ast}Q&-Q^{\ast}\\
0&I
\end{matrix}
\right)\left(\begin{matrix}
0&0\\
0&I
\end{matrix}
\right)\left(\begin{matrix}
TB_{1}^{-1}&0\\
-Q&I
\end{matrix}
\right)  \\
   & \times&\left(\begin{matrix}
(I+Q^{\ast}Q)^{-1}&(I+Q^{\ast}Q)^{-1}Q^{\ast}\\
0&I
\end{matrix}
\right)U^{\ast}\\&=&U\left(\begin{matrix}
Q^{\ast}Q&-Q^{\ast}\\
-Q&I
\end{matrix}
\right)\left(\begin{matrix}
(I+Q^{\ast}Q)^{-1}&(I+Q^{\ast}Q)^{-1}Q^{\ast}\\
0&I
\end{matrix}
\right)U^{\ast}\\&=&U\left(\begin{matrix}
I-(I+Q^{\ast}Q)^{-1}&-(I+Q^{\ast}Q)^{-1}Q^{\ast}\\
-Q(I+Q^{\ast}Q)^{-1}&I-Q(I+Q^{\ast}Q)^{-1}Q^{\ast}
\end{matrix}\right)U^{\ast}\\ &=&L_{B}^{\pi}=B^{\pi}.
\end{eqnarray*}
\end{proof}

Now, we investigate the stable perturbation bounds of the core-EP inverse. In order to simplify results, we again denote
	$E_{B}=L_{B}-A,$ $F_{B}=L_{B}-A^{\ast},$ $Y=(I+(A^{\scriptsize\textcircled{\tiny $\dagger$}})^{\ast}F_{B})^{-1}(A^{\scriptsize\textcircled{\tiny $\dagger$}})^{\ast}F_{B}A^{\pi}$ and $Z=A^{\pi}E_{B}A^{\scriptsize\textcircled{\tiny$\dagger$}}(I+E_{B}A^{\scriptsize\textcircled{\tiny $\dagger$}})^{-1}.$
\begin{theorem}\label{b5} Let $A\in\mathbb{C}^{n\times n}$ with ${\rm ind}(A)=k>0$ and let $B\in\mathbb{C}^{n\times n}$ with ${\rm ind}(B)=s.$ If $B$ is a stable perturbation of $A$, then
\begin{equation}\label{b66}
  \frac{\parallel B^{\scriptsize\textcircled{\tiny $\dagger$}}-A^{\scriptsize\textcircled{\tiny $\dagger$}}\parallel}{\parallel A^{\scriptsize\textcircled{\tiny $\dagger$}}\parallel}\leq\frac{\parallel W_{1}^{-1}\parallel\parallel W_{2}^{-1}\parallel}{\parallel A^{\scriptsize\textcircled{\tiny $\dagger$}}\parallel}(\parallel G_{1}\parallel+\parallel A^{\scriptsize\textcircled{\tiny $\dagger$}}\parallel\parallel G_{2}\parallel),
\end{equation}
where $G_{1}=A^{\scriptsize\textcircled{\tiny $\dagger$}}-(I+YZ-Z)A^{\scriptsize\textcircled{\tiny $\dagger$}}(I+Z^{\ast}Z-Z^{\ast})$ and $G_{2}=(I+E_{B}A^{\scriptsize\textcircled{\tiny $\dagger$}})^{-1}-I.$

Furthermore, if $\parallel Z\parallel<1$ and $\parallel YZ\parallel<1$, then
\begin{eqnarray*}
  \frac{\parallel B^{\scriptsize\textcircled{\tiny $\dagger$}}-A^{\scriptsize\textcircled{\tiny $\dagger$}}\parallel}{\parallel A^{\scriptsize\textcircled{\tiny $\dagger$}}\parallel} &\leq& \frac{1+\parallel Z\parallel}{\parallel A^{\scriptsize\textcircled{\tiny $\dagger$}}\parallel(1-\parallel Z\parallel)(1-\parallel YZ\parallel)}(\parallel G_{1}\parallel+\parallel A^{\scriptsize\textcircled{\tiny $\dagger$}}\parallel\parallel G_{2}\parallel) \\
  &\leq& \frac{1+\parallel Z\parallel}{(1-\parallel Z\parallel)(1-\parallel YZ\parallel)}(1+\alpha\beta+\parallel G_{2}\parallel),~~~~~~~~~~~~~~~~~~~~~~~~~~~(3.2)
\end{eqnarray*}
where $\alpha=1+\parallel Z\parallel+\parallel YZ\parallel$ and $\beta=1+\parallel Z\parallel+\parallel Z\parallel^{2}.$
\end{theorem}

\begin{proof}  By the proof of Lemma \ref{b3}, we know that $Z^{2}=0.$ Then $W_{1}=I+YZ-Z$ and $W_{2}=I+Z^{\ast}Z-Z^{\ast}.$ So, we have
\begin{eqnarray*}
 W_{1}^{-1}A^{\scriptsize\textcircled{\tiny $\dagger$}}W_{2}^{-1}-A^{\scriptsize\textcircled{\tiny $\dagger$}}&=&W_{1}^{-1}(A^{\scriptsize\textcircled{\tiny $\dagger$}}-W_{1}A^{\scriptsize\textcircled{\tiny $\dagger$}}W_{2})W_{2}^{-1}  \\
  &=&W_{1}^{-1}(A^{\scriptsize\textcircled{\tiny $\dagger$}}-(I+YZ-Z)A^{\scriptsize\textcircled{\tiny $\dagger$}}(I+Z^{\ast}Z-Z^{\ast}))W_{2}^{-1},
\end{eqnarray*} and
$$B^{\scriptsize\textcircled{\tiny $\dagger$}}- W_{1}^{-1}A^{\scriptsize\textcircled{\tiny $\dagger$}}W_{2}^{-1}= W_{1}^{-1}A^{\scriptsize\textcircled{\tiny $\dagger$}}((I+E_{B}A^{\scriptsize\textcircled{\tiny $\dagger$}})^{-1}-I)W_{2}^{-1}.$$
Denoting $$G_{1}=A^{\scriptsize\textcircled{\tiny $\dagger$}}-(I+YZ-Z)A^{\scriptsize\textcircled{\tiny $\dagger$}}(I+Z^{\ast}Z-Z^{\ast})$$ and $$G_{2}=(I+E_{B}A^{\scriptsize\textcircled{\tiny $\dagger$}})^{-1}-I,$$ we get
\begin{eqnarray*}
  B^{\scriptsize\textcircled{\tiny $\dagger$}}-A^{\scriptsize\textcircled{\tiny $\dagger$}} &=& W_{1}^{-1}A^{\scriptsize\textcircled{\tiny $\dagger$}}W_{2}^{-1}+W_{1}^{-1}A^{\scriptsize\textcircled{\tiny $\dagger$}}G_{2}W_{2}^{-1}+W_{1}^{-1}G_{1}W_{2}^{-1}-W_{1}^{-1}A^{\scriptsize\textcircled{\tiny $\dagger$}}W_{2}^{-1} \\
   &=&  W_{1}^{-1}(G_{1}+A^{\scriptsize\textcircled{\tiny $\dagger$}}G_{2})W_{2}^{-1}.
\end{eqnarray*}
Thus, $$\frac{\parallel B^{\scriptsize\textcircled{\tiny $\dagger$}}-A^{\scriptsize\textcircled{\tiny $\dagger$}}\parallel}{\parallel A^{\scriptsize\textcircled{\tiny $\dagger$}}\parallel}\leq\frac{\parallel W_{1}^{-1}\parallel\parallel W_{2}^{-1}\parallel}{\parallel A^{\scriptsize\textcircled{\tiny $\dagger$}}\parallel}(\parallel G_{1}\parallel+\parallel A^{\scriptsize\textcircled{\tiny $\dagger$}}\parallel\parallel G_{2}\parallel).$$

If $\parallel Z\parallel<1$ and$\parallel YZ\parallel<1$, then
$$\parallel W_{1}\parallel\leq1+\parallel YZ\parallel+\parallel Z\parallel,~~\parallel W_{2}\parallel\leq1+\parallel Z\parallel+\parallel Z\parallel^{2},$$
$$\parallel W_{1}^{-1}\parallel\leq\frac{1+\parallel Z\parallel}{1-\parallel YZ\parallel},~~\parallel W_{2}^{-1}\parallel\leq\frac{1+\parallel Z\parallel}{1-\parallel Z\parallel^{2}}=\frac{1}{1-\parallel Z\parallel}.$$
Setting $\alpha:=1+\parallel YZ\parallel+\parallel Z\parallel$ and $\beta:=1+\parallel Z\parallel+\parallel Z\parallel^{2},$ we obtain $$\parallel G_{1}\parallel\leq\parallel A^{\scriptsize\textcircled{\tiny $\dagger$}}\parallel+\parallel A^{\scriptsize\textcircled{\tiny $\dagger$}}\parallel\parallel W_{1}\parallel\parallel W_{2}\parallel\leq(1+\alpha\beta)\parallel A^{\scriptsize\textcircled{\tiny $\dagger$}}\parallel.$$  By substitution and simplification of inequality (\ref{b66}), we obtain inequality (3.2).
\end{proof}

\begin{theorem}\label{b6} Let $A\in\mathbb{C}^{n\times n}$ with ${\rm ind}(A)=k>0$ and let $B\in\mathbb{C}^{n\times n}$ with ${\rm ind}(B)=s.$ If $B$ is a stable perturbation of $A$ and $\parallel Z\parallel<1$, then
$$\parallel B^{\pi}-A^{\pi}\parallel\leq\frac{2\parallel Z\parallel}{1-\parallel Z\parallel}.$$
\end{theorem}

\begin{proof} We know that
		$E_{B}=L_{B}-A$ and $Z=A^{\pi}E_{B}A^{\scriptsize\textcircled{\tiny$\dagger$}}(I+E_{B}A^{\scriptsize\textcircled{\tiny $\dagger$}})^{-1}.$ Since $A^{\pi}Z^{\ast}=0$, $Z^{\ast}A^{\pi}=Z^{\ast}$ and $(I+Z^{\ast}Z)A^{\pi}=A^{\pi}=A^{\pi}(I+Z^{\ast}Z)^{-1}$, we have
\begin{eqnarray*}
  W_{2}A^{\pi}W_{2}^{-1} -A^{\pi}&=&(I-Z^{\ast})(I+Z^{\ast}Z)A^{\pi}(I+Z^{\ast}Z)^{-1}(I+Z^{\ast})- A^{\pi} \\
   &=& (I-Z^{\ast})A^{\pi}(I+Z^{\ast})- A^{\pi}\\
   &=&  (I-Z^{\ast})A^{\pi}-A^{\pi}=-Z^{\ast}.
\end{eqnarray*}
By the proof of Lemma \ref{b3}, it is easy to check that $A^{\pi}(I+E_{B}A^{\scriptsize\textcircled{\tiny $\dagger$}})^{-1}-A^{\pi}=-Z.$ Using that $A^{\pi}G_{2} = -Z,$ we obtain
\begin{eqnarray*}
  B^{\pi}-A^{\pi} &=&  W_{2}A^{\pi}(I+E_{B}A^{\scriptsize\textcircled{\tiny $\dagger$}})^{-1}W_{2}^{-1}-Z^{\ast}- W_{2}A^{\pi}W_{2}^{-1} \\
   &=& -Z^{\ast}-(I-Z^{\ast})Z(I+Z^{\ast}Z)^{-1}(I+Z^{\ast}).
\end{eqnarray*}
Since $\parallel Z\parallel<1,$ we have
$$\parallel B^{\pi}-A^{\pi}\parallel\leq\parallel Z\parallel+\frac{\parallel Z\parallel(1+\parallel Z\parallel)}{1-\parallel Z\parallel}=\frac{2\parallel Z\parallel}{1-\parallel Z\parallel}.$$
\end{proof}

\begin{remark}\label{b7} Let $B$ be a stable perturbation of $A$. Then the following two statements hold:
\begin{itemize}
\item [{\rm (i)}] If $\parallel E_{B}A^{\scriptsize\textcircled{\tiny $\dagger$}}\parallel+\parallel A^{\pi}E_{B}A^{\scriptsize\textcircled{\tiny $\dagger$}}\parallel<1$, then $\parallel Z\parallel<1$;
\item [{\rm (ii)}] If ${\rm max}\{ \parallel E_{B}A^{\scriptsize\textcircled{\tiny $\dagger$}}\parallel,~\parallel (A^{\scriptsize\textcircled{\tiny $\dagger$}})^{\ast}F_{B}\parallel\}<\frac{1}{1+\sqrt{\parallel A^{\pi}\parallel}},$ then $\parallel YZ\parallel<1.$
\end{itemize}
In fact, by the expressions of $Y$ and $Z$, we have
$$\parallel Z\parallel\leq\frac{\parallel A^{\pi}E_{B}A^{\scriptsize\textcircled{\tiny $\dagger$}}\parallel}{1-\parallel E_{B}A^{\scriptsize\textcircled{\tiny $\dagger$}}\parallel}<1,$$
$$\parallel YZ\parallel\leq\frac{\parallel (A^{\scriptsize\textcircled{\tiny $\dagger$}})^{\ast}F_{B}\parallel\parallel A^{\pi}E_{B}A^{\scriptsize\textcircled{\tiny $\dagger$}}\parallel}{(1-\parallel (A^{\scriptsize\textcircled{\tiny $\dagger$}})^{\ast}F_{B}\parallel)(1-\parallel E_{B}A^{\scriptsize\textcircled{\tiny $\dagger$}}\parallel)}<\frac{\parallel A^{\pi}\parallel(\frac{1}{1+\sqrt{\parallel A^{\pi}\parallel}})^{2}}{(1-\frac{1}{1+\sqrt{\parallel A^{\pi}\parallel}})^{2}}=1.$$
\end{remark}

\section{Acute perturbation of the core-EP inverse}\label{c}
In this section, we investigate the acute perturbation of the core-EP inverse. 

Firstly, we present the definition of acute perturbation.
\begin{definition}\label{c1} Let $A\in\mathbb{C}^{n\times n}$. A matrix $B\in C^{n \times n}$ is called an acute perturbation related to the core-EP inverse (or in short, acute perturbation) if $\parallel E_{B}\parallel$ is small and the spectral radius $\rho( B^{\pi}-A^{\pi})<1$.
\end{definition}

Now, we give an upper bound for the spectral radius $\rho( B^{\pi}-A^{\pi})$. In order to simplify results, we again denote
	$E_{B}=L_{B}-A$ and $Z=A^{\pi}E_{B}A^{\scriptsize\textcircled{\tiny$\dagger$}}(I+E_{B}A^{\scriptsize\textcircled{\tiny $\dagger$}})^{-1}.$
\begin{theorem}\label{c2} Let $A ,B\in\mathbb{C}^{n\times n}$ with ${\rm ind}(A)=k>0$, ${\rm ind}(B)=s$ and ${\rm rk}(A^{k})={\rm rk}(B^{s}).$ If the perturbation $E_{B}$ satisfies $\parallel E_{B}A^{\scriptsize\textcircled{\tiny $\dagger$}}\parallel<\frac{1}{1+\sqrt{2\parallel A^{\pi}\parallel}}$, then the following conditions hold:
\begin{itemize}
\item [{\rm (i)}] $\rho(Z^{\ast}Z)<\frac{1}{2}$;
\item [{\rm (ii)}] $\rho(BB^{\scriptsize\textcircled{\tiny $\dagger$}}(I-AA^{\scriptsize\textcircled{\tiny $\dagger$}}))=\rho(AA^{\scriptsize\textcircled{\tiny $\dagger$}}(I-BB^{\scriptsize\textcircled{\tiny $\dagger$}}))\leq\frac{\rho(Z^{\ast}Z)}{1-\rho(Z^{\ast}Z)}$;
\item [{\rm (iii)}]$(\rho( B^{\pi}-A^{\pi}))^{2}=\rho(AA^{\scriptsize\textcircled{\tiny $\dagger$}}(I-BB^{\scriptsize\textcircled{\tiny $\dagger$}}))=\rho( (B^{\pi}-A^{\pi})^{2})<1$;
\item [{\rm (iv)}] $I-BB^{\scriptsize\textcircled{\tiny $\dagger$}}A^{\pi}$, $I-AA^{\scriptsize\textcircled{\tiny $\dagger$}}B^{\pi}$, $I-(B^{\pi}-A^{\pi})$ and $I-(B^{\pi}-A^{\pi})^{2}$ are nonsingular.
\end{itemize}
 \end{theorem}

\begin{proof} Let $A$ be as in Lemma \ref{a0}.  Since $\rho(E_B A^{\scriptsize\textcircled{\tiny $\dagger$}})\leqslant\parallel E_{B}A^{\scriptsize\textcircled{\tiny $\dagger$}}\parallel<\frac{1}{1+\sqrt{2\parallel A^{\pi}\parallel}}<1$, we obtain that $-1$ is not an eigenvalue of $E_{B}A^{\scriptsize\textcircled{\tiny $\dagger$}}$. Therefore, $0$ is not an eigenvalue of $I+E_{B}A^{\scriptsize\textcircled{\tiny $\dagger$}}$, that is $I+E_{B}A^{\scriptsize\textcircled{\tiny $\dagger$}}$ is nonsingular. Combining ${\rm rk}(A^{k})={\rm rk}(B^{s})$, by Lemma \ref{b2} ${\rm (vi)}$ and ${\rm (vii)}$, it is easy to obtain $L_{B}=U\left(\begin{matrix}
B_{1}&B_{1}P\\
QB_{1}&QB_{1}P
\end{matrix}
\right)U^{\ast},$ for some matrices $B_{1}$, $P$ and $Q$ such that $B_{1}$ and $I+PQ$ are nonsingular.

${\rm (i)}:$ Since $\parallel E_{B}A^{\scriptsize\textcircled{\tiny $\dagger$}}\parallel<\frac{1}{1+\sqrt{2\parallel A^{\pi}\parallel}}$, we have
\begin{eqnarray*}
  \rho(Z^{\ast}Z) &\leq&\parallel Z^{\ast}Z\parallel\leq \left(\frac{\parallel E_{B}A^{\scriptsize\textcircled{\tiny $\dagger$}}\parallel}{1-\parallel E_{B}A^{\scriptsize\textcircled{\tiny $\dagger$}}\parallel}\right)\left(\frac{\parallel A^{\pi}E_{B}A^{\scriptsize\textcircled{\tiny $\dagger$}}\parallel}{1-\parallel E_{B}A^{\scriptsize\textcircled{\tiny $\dagger$}}\parallel}\right) \\
  &\leq&  \frac{\parallel A^{\pi}\parallel(\parallel E_{B}A^{\scriptsize\textcircled{\tiny $\dagger$}}\parallel)^{2}}{(1-\parallel E_{B}A^{\scriptsize\textcircled{\tiny $\dagger$}}\parallel)^{2}}\\
  &=& \frac{\parallel A^{\pi}\parallel}{(\frac{1}{\parallel E_{B}A^{\scriptsize\textcircled{\tiny $\dagger$}}\parallel}-1)^{2}}\\
   &<& \frac{\parallel A^{\pi}\parallel}{(1+\sqrt{2\parallel A^{\pi}\parallel}-1)^{2}}=\frac{1}{2}.
\end{eqnarray*}

 ${\rm (ii)}:$ By Lemma \ref{a00} and Lemma \ref{a2}, we know that
 \begin{eqnarray*}
   BB^{\scriptsize\textcircled{\tiny $\dagger$}} &=& L_{B}L_{B}^{\scriptsize\textcircled{\tiny \#}}= U\left(\begin{matrix}
(I+Q^{\ast}Q)^{-1}&(I+Q^{\ast}Q)^{-1}Q^{\ast}\\
Q(I+Q^{\ast}Q)^{-1}&Q(I+Q^{\ast}Q)^{-1}Q^{\ast}
\end{matrix}
\right)U^{\ast}\\
  &=& U\left(\begin{matrix}
I&0\\
Q&I
\end{matrix}
\right)\left(\begin{matrix}
I&(I+Q^{\ast}Q)^{-1}Q^{\ast}\\
0&0
\end{matrix}
\right)\left(\begin{matrix}
I&0\\
-Q&I
\end{matrix}
\right)U^{\ast}.
 \end{eqnarray*}
  Then
\begin{eqnarray*}
   BB^{\scriptsize\textcircled{\tiny $\dagger$}}(I-AA^{\scriptsize\textcircled{\tiny $\dagger$}})&=& L_{B}L_{B}^{\scriptsize\textcircled{\tiny \#}}(I-AA^{\scriptsize\textcircled{\tiny $\dagger$}})=U\left(\begin{matrix}
I&0\\
Q&I
\end{matrix}
\right) \\
   &\times& \left(\begin{matrix}
I&(I+Q^{\ast}Q)^{-1}Q^{\ast}\\
0&0
\end{matrix}
\right)\left(\begin{matrix}
I&0\\
-Q&I
\end{matrix}
\right)\left(\begin{matrix}
0&0\\
0&I
\end{matrix}
\right)\left(\begin{matrix}
I&0\\
Q&I
\end{matrix}
\right)\left(\begin{matrix}
I&0\\
-Q&I
\end{matrix}
\right)U^{\ast} \\
   &=& U \left(\begin{matrix}
I&0\\
Q&I
\end{matrix}
\right)\left(\begin{matrix}
(I+Q^{\ast}Q)^{-1}Q^{\ast}Q&(I+Q^{\ast}Q)^{-1}Q^{\ast}\\
0&0
\end{matrix}
\right)\left(\begin{matrix}
I&0\\
-Q&I
\end{matrix}
\right)U^{\ast},
\end{eqnarray*}
\begin{eqnarray*}
  &&AA^{\scriptsize\textcircled{\tiny $\dagger$}}(I-BB^{\scriptsize\textcircled{\tiny $\dagger$}}) = U\left(\begin{matrix}
I&0\\
0&0
\end{matrix}
\right)\left(\begin{matrix}
I&0\\
Q&I
\end{matrix}
\right)\left(\begin{matrix}
0&-(I+Q^{\ast}Q)^{-1}Q^{\ast}\\
0&I
\end{matrix}
\right)\left(\begin{matrix}
I&0\\
-Q&I
\end{matrix}
\right)U^{\ast}\\&&=U\left(\begin{matrix}
I&0\\
Q&I
\end{matrix}
\right)\left(\begin{matrix}
I&0\\
-Q&I
\end{matrix}
\right)\left(\begin{matrix}
I&0\\
0&0
\end{matrix}
\right)\left(\begin{matrix}
I&0\\
Q&I
\end{matrix}
\right)\left(\begin{matrix}
0&-(I+Q^{\ast}Q)^{-1}Q^{\ast}\\
0&I
\end{matrix}
\right)\left(\begin{matrix}
I&0\\
-Q&I
\end{matrix}
\right)U^{\ast}  \\&&
   = U\left(\begin{matrix}
I&0\\
Q&I
\end{matrix}
\right)\left(\begin{matrix}
0&-(I+Q^{\ast}Q)^{-1}Q^{\ast}\\
0&Q(I+Q^{\ast}Q)^{-1}Q^{\ast}
\end{matrix}
\right)\left(\begin{matrix}
I&0\\
-Q&I
\end{matrix}
\right)U^{\ast}.
\end{eqnarray*}
By a direct computation, we obtain
\begin{eqnarray*}
  \rho(BB^{\scriptsize\textcircled{\tiny $\dagger$}}(I-AA^{\scriptsize\textcircled{\tiny $\dagger$}})) &=& \rho((I+Q^{\ast}Q)^{-1}Q^{\ast}Q)=\rho(Q(I+Q^{\ast}Q)^{-1}Q^{\ast}) \\
   &=& \rho(AA^{\scriptsize\textcircled{\tiny $\dagger$}}(I-BB^{\scriptsize\textcircled{\tiny $\dagger$}}))=\rho((I+Z^{\ast}Z)^{-1}Z^{\ast}Z) \\
   &\leq&\rho((I+Z^{\ast}Z)^{-1})\rho(Z^{\ast}Z) \leq\frac{\rho(Z^{\ast}Z)}{1-\rho(Z^{\ast}Z)} <1, ~{\rm by ~condition ~(i)}.
\end{eqnarray*}

${\rm (iii)}:$ By Lemma \ref{a2}, we have
\begin{eqnarray*}
  B^{\pi}-A^{\pi}&=& U\left(\begin{matrix}
I&0\\
Q&I
\end{matrix}
\right)\left(\begin{matrix}
0&-(I+Q^{\ast}Q)^{-1}Q^{\ast}\\
0&I
\end{matrix}
\right)\left(\begin{matrix}
I&0\\
-Q&I
\end{matrix}
\right)U^{\ast}-U\left(\begin{matrix}
0&0\\
0&I
\end{matrix}
\right)U^{\ast} \\
   &=& U\left(\begin{matrix}
I&0\\
Q&I
\end{matrix}
\right)\left(\begin{matrix}
0&-(I+Q^{\ast}Q)^{-1}Q^{\ast}\\
-Q&0
\end{matrix}
\right)\left(\begin{matrix}
I&0\\
-Q&I
\end{matrix}
\right)U^{\ast},
\end{eqnarray*}
$$(B^{\pi}-A^{\pi})^{2}=U\left(\begin{matrix}
I&0\\
Q&I
\end{matrix}
\right)\left(\begin{matrix}
(I+Q^{\ast}Q)^{-1}Q^{\ast}Q&0\\
0&Q(I+Q^{\ast}Q)^{-1}Q^{\ast}
\end{matrix}
\right)\left(\begin{matrix}
I&0\\
-Q&I
\end{matrix}
\right)U^{\ast}.$$
Since $\rho((I+Q^{\ast}Q)^{-1}Q^{\ast}Q)=\rho(Q(I+Q^{\ast}Q)^{-1}Q^{\ast})$, by conditions ${\rm (i)}$ and ${\rm (ii)}$, we get that
\begin{eqnarray*}
 (\rho(B^{\pi}-A^{\pi}))^{2}  &=&  \rho((B^{\pi}-A^{\pi})^{2})=\rho(Q(I+Q^{\ast}Q)^{-1}Q^{\ast})\\
   &=& \rho(AA^{\scriptsize\textcircled{\tiny $\dagger$}}(I-BB^{\scriptsize\textcircled{\tiny $\dagger$}}))<1.
\end{eqnarray*}

${\rm (iv)}:$ It is clear by conditions ${\rm (i)}$, ${\rm (ii)}$ and ${\rm (iii)}$.
\end{proof}

From Definition \ref{c1} and Theorem \ref{c2}, we have the following result.
\begin{corollary}\label{c3} Let $A ,B\in\mathbb{C}^{n\times n}$ with ${\rm ind}(A)=k>0$ and ${\rm ind}(B)=s$. If ${\rm rk}(A^{k})={\rm rk}(B^{s})$ and $\parallel E_{B}A^{\scriptsize\textcircled{\tiny $\dagger$}}\parallel<\frac{1}{1+\sqrt{2\parallel A^{\pi}\parallel}}$, then $B$ is an acute perturbation of $A$.
\end{corollary}

\begin{remark}\label{c33} Since $A^{\pi}$ and $B^{\pi}$ are orthogonal projectors, we have $$\parallel B^{\pi}-A^{\pi}\parallel=(\rho((B^{\pi}-A^{\pi})^{\ast}(B^{\pi}-A^{\pi})))^{\frac{1}{2}}=(\rho((B^{\pi}-A^{\pi})^{2}))^{\frac{1}{2}}=\rho(B^{\pi}-A^{\pi}).$$
\end{remark}

Now, we present sufficient and necessary conditions for acute perturbation related to the core-EP inverse by using the results obtained for the stable perturbation about the core-EP inverse in Lemma \ref{b2}.
\begin{theorem}\label{c4} Let $A ,B\in\mathbb{C}^{n\times n}$ with ${\rm ind}(A)=k>0$ and ${\rm ind}(B)=s$. Then the following conditions are equivalent:
\begin{itemize}
\item [{\rm (i)}] $B$ is an acute perturbation of $A$ with respect to the core-EP inverse;
\item [{\rm (ii)}] $B$ satisfies condition $(C_{s,\ast})$;
\item [{\rm (iii)}] $B$ is a stable perturbation of $A$ with respect to the core-EP inverse;
\item [{\rm (iv)}] ${\rm rk}(A^{k})={\rm rk}(B^{s})={\rm rk}((A^{k})^{\ast}L_{B}A^{k})$.
\end{itemize}
\end{theorem}

\begin{proof} By Lemma \ref{b2}, it is easy to know that ${\rm (ii)}\Leftrightarrow{\rm (iii)}\Leftrightarrow{\rm (iv)}$.

${\rm (i)}\Rightarrow{\rm (ii)}:$ Suppose that $B$ is an acute perturbation of $A$ with respect to the core-EP inverse. By Definition \ref{c1}, we have $\rho(B^{\pi}-A^{\pi})<1$. If $R(B^{s})\cap N((A^{k})^{\ast})\neq\{0\}$, then there exists a nonzero vector (see \cite{FLT}) $x\in\mathbb{C}^{n}$ such that
$$BB^{\scriptsize\textcircled{\tiny $\dagger$}}x=x,~~AA^{\scriptsize\textcircled{\tiny $\dagger$}}x=0.$$
That is, $(BB^{\scriptsize\textcircled{\tiny $\dagger$}}-AA^{\scriptsize\textcircled{\tiny $\dagger$}})x=x.$ Hence $\rho(B^{\pi}-A^{\pi})=\rho(BB^{\scriptsize\textcircled{\tiny $\dagger$}}-AA^{\scriptsize\textcircled{\tiny $\dagger$}})\geq1$. This produces a contradiction.
Thus $R(B^{s})\cap N((A^{k})^{\ast})=\{0\}$. Similarly, we obtain $R(A^{k})\cap N((B^{s})^{\ast})=\{0\}$.

${\rm (iv)}\Rightarrow{\rm (i)}:$ Suppose that $\parallel E_{B}\parallel=\parallel L_{B}-A\parallel$ is small such that $\parallel E_{B}A^{\scriptsize\textcircled{\tiny $\dagger$}}\parallel<\frac{1}{1+\sqrt{2\parallel A^{\pi}\parallel}}$. Since ${\rm rk}(A^{k})={\rm rk}(B^{s})={\rm rk}((A^{k})^{\ast}L_{B}A^{k})$, by Theorem \ref{c2}, we obtain $\rho(B^{\pi}-A^{\pi})<1.$
\end{proof}

\begin{remark}\label{c5} If $B$ is not an acute perturbation of $A$ with respect to the core-EP inverse and ${\rm rk}(A^{k})<{\rm rk}(B^{s})$, then $\rho(B^{\pi}-A^{\pi})=\parallel B^{\pi}-A^{\pi}\parallel\geq1.$ In fact, since $\mathbb{C}^{n}=R(A^{k})\oplus N((A^{k})^{\ast})$ and ${\rm rk}(A^{k})<{\rm rk}(B^{s})$, we have
$R(B^{s})\cap N((A^{k})^{\ast})\neq\{0\}$. There exists nonzero $x\in\mathbb{C}^{n}$ such that
$$BB^{\scriptsize\textcircled{\tiny $\dagger$}}x=x,~~AA^{\scriptsize\textcircled{\tiny $\dagger$}}x=0.$$ Assuming, without loss of generality, $\parallel x\parallel=1$, we arrive at
$$1=\parallel x\parallel=\parallel (BB^{\scriptsize\textcircled{\tiny $\dagger$}}-AA^{\scriptsize\textcircled{\tiny $\dagger$}})x\parallel\leq\parallel BB^{\scriptsize\textcircled{\tiny $\dagger$}}-AA^{\scriptsize\textcircled{\tiny $\dagger$}}\parallel=\parallel B^{\pi}-A^{\pi}\parallel.$$
By Remark \ref{c33}. we have $\rho(B^{\pi}-A^{\pi})=\parallel B^{\pi}-A^{\pi}\parallel\geq1$.
\end{remark}

Finally, we give an example to check properties of acute perturbation obtained in Theorem \ref{c2}. 
\begin{example}\label{c5} \emph{\cite[Example 5.5]{ZCT}} Let $A=\left(\begin{matrix}
1&2&\frac{1}{10}&\frac{1}{10}\\
2&1&0&0\\
0&0&0&1\\
0&0&0&0
\end{matrix}
\right)$ with ${\rm ind}(A)=2$. Then $A^{\scriptsize\textcircled{\tiny $\dagger$}}=\left(\begin{array}{rrcc}
-\frac{1}{3}&\frac{2}{3}&0&0\\
\frac{2}{3}&-\frac{1}{3}&0&0\\
0&0&0&0\\
0&0&0&0
\end{array}
\right)$. Set $B\in\mathbb{C}^{4\times 4}$ with ${\rm ind}(B)=s$, where $0<s<4$. By the equivalence of condition ${\rm (i)}$ and condition ${\rm (vi)}$ in Lemma \ref{b2}, we set $B_{1}=\left(\begin{array}{cccr}
1&2\\
2&1
\end{array}
\right)$, $P=\left(\begin{array}{cccr}
\frac{1}{10}&\frac{1}{10}\\
0&0
\end{array}
\right)$ and $Q=\left(\begin{array}{cccr}
\frac{1}{5}&0\\
\frac{1}{10}&0
\end{array}
\right)$. By using MATLAB, we obtain $$L_{B}=\left(\begin{matrix}
1&2&\frac{1}{10}&\frac{1}{10}\\
2&1&\frac{1}{5}&\frac{1}{5}\\
\frac{1}{5}&\frac{2}{5}&\frac{1}{50}&\frac{1}{50}\\
\frac{1}{10}&\frac{1}{5}&\frac{1}{100}&\frac{1}{100}
\end{matrix}
\right),~B^{\scriptsize\textcircled{\tiny $\dagger$}}=L_{B}^{\scriptsize\textcircled{\tiny \#}}=\left(\begin{array}{rcrr}
-\frac{409}{1327}&\frac{200}{309}&-\frac{400}{6489}&-\frac{200}{6489}\\
\frac{40}{63}&-\frac{1}{3}&\frac{8}{63}&\frac{4}{63}\\
-\frac{400}{6489}&\frac{40}{309}&-\frac{80}{6489}&-\frac{40}{6489}\\
-\frac{200}{6489}&\frac{20}{309}&-\frac{40}{6489}&-\frac{20}{6489}
\end{array}
\right),$$
$$A^{\pi}=\left(\begin{matrix}
0&0&0&0\\
0&0&0&0\\
0&0&1&0\\
0&0&0&1
\end{matrix}
\right),~B^{\pi}=\left(\begin{array}{rcrr}
\frac{1}{21}&0&-\frac{4}{21}&-\frac{2}{21}\\
0&0&0&0\\
-\frac{4}{21}&0&\frac{101}{105}&-\frac{2}{105}\\
-\frac{2}{21}&0&-\frac{2}{105}&\frac{104}{105}
\end{array}
\right),$$
$$E_{B}=L_{B}-A=\left(\begin{array}{cccr}
0&0&0&0\\
0&0&\frac{1}{5}&\frac{1}{5}\\
\frac{1}{5}&\frac{2}{5}&\frac{1}{50}&-\frac{49}{50}\\
\frac{1}{10}&\frac{1}{5}&\frac{1}{100}&\frac{1}{100}
\end{array}\right),~F_{B}=L_{B}-A^{\ast}=\left(\begin{array}{ccrc}
0&0&\frac{1}{10}&\frac{1}{10}\\
0&0&\frac{1}{5}&\frac{1}{5}\\
\frac{1}{10}&\frac{2}{5}&\frac{1}{50}&\frac{1}{50}\\
0&\frac{1}{5}&-\frac{99}{100}&\frac{1}{100}
\end{array}
\right).$$
In this case, we have ${\rm rk}(A^{2})=2={\rm rk}(L_{B})={\rm rk}(B^{s})$, $\parallel E_{B}A^{\scriptsize\textcircled{\tiny $\dagger$}}\parallel=\frac{646}{2889}<1$ and $\frac{1}{1+\sqrt{2\parallel A^{\pi}\parallel}}=\frac{408}{985}$. Then $\parallel E_{B}A^{\scriptsize\textcircled{\tiny $\dagger$}}\parallel<\frac{1}{1+\sqrt{2\parallel A^{\pi}\parallel}}.$ By Theorem \ref{c2}, we compute $\rho(Z^{\ast}Z)=\frac{1}{20}<\frac{1}{2}$, $\rho(BB^{\scriptsize\textcircled{\tiny $\dagger$}}(I-AA^{\scriptsize\textcircled{\tiny $\dagger$}}))=\frac{1}{21}=\rho(AA^{\scriptsize\textcircled{\tiny $\dagger$}}(I-BB^{\scriptsize\textcircled{\tiny $\dagger$}}))<\frac{1}{20}$, $\rho(B^{\pi}-A^{\pi})=\frac{769}{3524}$, $(\rho(B^{\pi}-A^{\pi}))^{2}= \frac{1}{21}=\rho((B^{\pi}-A^{\pi})^{2})<\frac{1}{20}.$
\end{example}

\section{\bf Applications}
 In \cite{M}, Ma studied optimal perturbation bounds for the core inverse. In \cite{MS}, Ma et al. investigated optimal perturbation  bounds for core-EP inverses. The conditions required to calculate optimal perturbation bounds for generalized inverses are relatively simple in our case. In fact, it is a special case of the acute (or stable) perturbation.

According to {\cite[Theorem 5.1]{MS}}, we have the following theorem.
\begin{theorem}\label{d1} Let $A ,B\in\mathbb{C}^{n\times n}$ with ${\rm ind}(A)=k>0$ and ${\rm ind}(B)=s$. Denoting $E_{B}=L_{B}-A$. If $E_{B}$ satisfies $AA^{\scriptsize\textcircled{\tiny $\dagger$}}E_{B}=E_{B}$ and $\parallel E_{B}A^{\scriptsize\textcircled{\tiny $\dagger$}}\parallel<1$, then
$$B^{\scriptsize\textcircled{\tiny $\dagger$}}=A^{\scriptsize\textcircled{\tiny $\dagger$}}(I+E_{B}A^{\scriptsize\textcircled{\tiny $\dagger$}})^{-1}~{\rm and}~BB^{\scriptsize\textcircled{\tiny $\dagger$}}=AA^{\scriptsize\textcircled{\tiny $\dagger$}}.$$
Moreover, $B$ is an acute perturbation with respect to the core-EP inverse of $A$ and $B$ is a stable perturbation with respect to the core-EP inverse of $A.$
\end{theorem}

\begin{proof} Similar to the proof of \cite[Theorem 5.1]{MS}, it is easy to check that ${\rm rk}(A^{k})={\rm rk}(L_{B})={\rm rk}(B^{s})$ and $\rho(B^{\pi}-A^{\pi})=\rho(BB^{\scriptsize\textcircled{\tiny $\dagger$}}-AA^{\scriptsize\textcircled{\tiny $\dagger$}})=0.$ So, $B$ satisfies condition $(C_{s,\ast})$. By Theorem \ref{c4}, we get that $B$ is an acute perturbation of $A$ and $B$ is a stable perturbation of $A$.
\end{proof}


In particular, we have the following corollaries when ${\rm ind}(A)=1.$
\begin{corollary}\label{d3} Let $A ,B\in\mathbb{C}^{n\times n}$ with ${\rm ind}(A)=1$ and ${\rm ind}(B)=s$. Denoting $E_{B}=L_{B}-A$. If $E_{B}$ satisfies $AA^{\scriptsize\textcircled{\tiny \#}}E_{B}=E_{B}$ and $\parallel E_{B}A^{\scriptsize\textcircled{\tiny \#}}\parallel<1$, then
$$B^{\scriptsize\textcircled{\tiny \#}}=A^{\scriptsize\textcircled{\tiny \#}}(I+E_{B}A^{\scriptsize\textcircled{\tiny \#}})^{-1}~{\rm and}~BB^{\scriptsize\textcircled{\tiny \#}}=AA^{\scriptsize\textcircled{\tiny \#}}.$$
Moreover, $B$ is an acute perturbation with respect to the core inverse of $A$ and $B$ is a stable perturbation with respect to the core inverse of $A.$
\end{corollary}

\begin{remark}\label{d5} If we change $E_{B}=L_{B}-A$, $AA^{\scriptsize\textcircled{\tiny $\dagger$}}E_{B}=E_{B}$ and $\parallel E_{B}A^{\scriptsize\textcircled{\tiny $\dagger$}}\parallel<1$ to $F_{B}=L_{B}-A^{\ast}$, $AA^{\scriptsize\textcircled{\tiny $\dagger$}}F_{B}=F_{B}$ and $\parallel (A^{\scriptsize\textcircled{\tiny $\dagger$}})^{\ast}F_{B}\parallel<1$ in conditions of Theorem \ref{d1}, respectively, then we have
$$B^{\scriptsize\textcircled{\tiny $\dagger$}}=(I+(A^{\scriptsize\textcircled{\tiny $\dagger$}})^{\ast}F_{B})^{-1}(A^{\scriptsize\textcircled{\tiny $\dagger$}})^{\ast}~{\rm and}~BB^{\scriptsize\textcircled{\tiny $\dagger$}}=AA^{\scriptsize\textcircled{\tiny $\dagger$}}.$$
Moreover, $B$ is an acute perturbation with respect to the core-EP inverse of $A.$ 

This is the dual case of Theorem \ref{d1}.
\end{remark}
\bigskip

\centerline {\bf Disclosure Statement}

Nothing to report.

\centerline {\bf Acknowledgments} 

We would like to thank the Editor and both Referees for their valuable comments and suggestions which improve the readability of the paper.

This research is supported by the National Natural Science Foundation of China (No. 12101315, 12101539,12171083), the China Scholarship Council (File No. 201906090122),  Natural Science Foundation of Jiangsu Higher Education Institutions of China (21KJB110004), the Qing Lan Project of Jiangsu Province. The third author is partially supported by 
Ministerio de Econom\'{\i}a, Industria y Competitividad of Spain (Grant Red de Excelencia RED2022-134176-T), partially supported by Universidad de Buenos Aires, Argentina (EXP-UBA: 13.019/2017, 20020170100350BA), and partially supported by  Universidad Nacional de La Pampa, Argentina, Facultad de Ingenier\'{i}a (Resoluci\'on del CD N° 135/19).

\end{document}